\documentclass{article}

\usepackage[left=3cm,right=3cm,top=3cm,bottom=3cm,bindingoffset=0cm]{geometry}
\usepackage{amsmath,amsthm,amssymb}

\usepackage{hyperref}
\hypersetup{
    colorlinks,
    citecolor=blue,
    filecolor=blue,
    linkcolor=blue,
    urlcolor=blue
}

\usepackage{xcolor}
\newtheorem{theorem}{Theorem}

\newtheorem{corollary}{Corollary}

\newtheorem{lemma}{Lemma}
\newtheorem{definition}{Definition}

\begin{document}

\large

\author{\large Ilnur I. Batyrshin\thanks{%
batyrshin@gmail.com.
} \\
Kazan Federal University, Kazan, Russia}

\title{Countable strict reverse mathematics}
\date{August 19, 2022}
\maketitle

\begin{abstract}
\large
We investigate subsystems $COM_{fcn}$, $COMI_{fcn}$ and $PRA_{fcn}$ of the elementary theory of functions $ETF$, the base theory for countable strict reverse mathematics. We show that inductions on any variable for unary, binary and ternary functions are pairwise equivalent over $COM_{fcn}$. We prove that weakened primitive recursion axiom $WPRA$ is equivalent to primitive recursion axiom $PRA$ over $COMI_{fcn}$. We show that permutation axiom and minimization axioms $MIN^1$, $MIN^2$, $MIN^3$ are pairwise equivalent over $PRA_{fcn}$. Thus, we present several equivalent axiomatizations of $ETF$.
\end{abstract}

\section{Introduction}

\

Reverse mathematics initiated by Harvey Friedman \cite{Fr75}, \cite{Fr76} and shaped by Simpson and his several PhD students is a flourishing branch of contemporary mathematics that needs no introduction. The first and principal book on reverse mathematics by Simpson \cite{Si99}, \cite{Si09} is complemented wonderfully by more recent texts by Hirschfeldt \cite{Hi14}, Stillwell \cite{Sti18} as well as fresh works by Monin and Patey \cite{MoPa} and Dzhafarov and Mummert \cite{DzMu}; numerous papers have demonstrated that reverse mathematics is a fine tool to classify classical mathematical theorems from proof theoretic and epistemological viewpoints.

Strict reverse mathematics is a program in its infancy that seeks to calibrate the strength of mathematical theorems in terms of strictly mathematical axioms, avoiding the need for coding, in contrast to reverse mathematics that uses some purely logical axiom schemes and relies heavily on coding. This program was initiated by Friedman in \cite{Fr76} and \cite{Fr09} and further promoted in his recent preprint \cite{Fr21}, but its systematic development has not yet begun.

The main strict reverse mathematics requirements are that all base theories and target theories be strictly mathematical and that all reversed statements be strictly mathematical and unaltered by coding. These ambitious requirements have the other side of the coin, namely that unlike reverse mathematics there is no single language, no single formal system, no single way of doing things in strict reverse mathematics. While the setting for reverse mathematics is the language of second order arithmetic $L_2$, there are several fundamentally different formal systems in strict reverse mathematics even by now. In \cite{Fr76} the four-sorted language with numbers and functions of arities 1, 2, 3 was introduced. In \cite{Fr09} several logical systems were considered, one of them with two sorts, corresponding to the integers and the finite sequences of integers. In \cite{Fr21} the eight-sorted language was used to construe reverse mathematics as the special case of strict reverse mathematics. ``It is a feature of strict reverse mathematics, not a bug'', Friedman commented on this situation (personal communication). In strict reverse mathematics for each area of mathematics there should be a special language and a special formal system with basic concepts of this area as primitives.

Strict reverse mathematics can be subdivided into two major areas: countable strict reverse mathematics and finite strict reverse mathematics \cite{Fr22}. In the former the objects referred to in the base theory are countable, in the latter the objects referred to in the base theory are finite.

In this paper we study the formal system of elementary theory of functions $ETF$ introduced and taken as the base theory for countable strict reverse mathematics in \cite{Fr21}.

The language of elementary theory of functions $L_{fcn}$ is a four-sorted language that has four distinct sorts of variables which are intended to range over four different kinds of object. Variables of the first sort are number variables, are denoted by $i, j, k, m, n, ...$, and are intended to range over the set $\omega=\{0, 1, 2, ...\}$ of all natural numbers. Variables of the second sort are functions variables of arity 1, are denoted by $f_0^1, f_1^1, f_2^1...$, and are intended to range over all unary functions on $\omega$. Variables of the third sort are functions variables of arity 2, are denoted by $f_0^2, f_1^2, f_2^2...$, and are intended to range over all binary functions on $\omega$. Variables of the fourth sort are functions variables of arity 3, are denoted by $f_0^3, f_1^3, f_2^3...$, and are intended to range over all ternary functions on $\omega$. The language $L_{fcn}$ has also the constant symbol $0$ of sort $\omega$ and the unary function symbol $S$ of the second sort which is intended to denote the unary successor function.

The terms of $L_{fcn}$ are built up in the usual way, namely they are number variables, the constant symbol $0$, and $S(t)$, $f_i^1(t)$, $f_j^2(t,q)$, $f_k^3(t,q,r)$ whenever $t,q,r$ are terms and $f_i^1, f_j^2, f_k^3$ are functions variables of corresponding arity. Atomic formulas are $s=t$ where $s$ and $t$ are terms. Formulas are built up as usual from atomic formulas by means of propositional connectives $\wedge, \vee, \neg, \rightarrow, \leftrightarrow$ and quantifiers $\forall,\exists$ over numbers and functions.

\begin{definition} [Friedman \cite{Fr21}] $L_{fcn}$ is defined to be the language of elementary theory of functions as described above.
\end{definition}

For convenience, we will omit indices and use $f,g,h,f',g',h'$ as function variables when their arities are clear from the context. We also adopt the uniqueness quantifier $\exists!$ in a usual way and shall use parentheses to indicate grouping as is customary in mathematical logic textbooks.

The language $L_{fcn}$ allows to present the strictly mathematical system in numbers and functions ETF that turns out to be logically equivalent to the well-known in reverse mathematics system $RCA_0$ when $RCA_0$ is put into $L_{fcn}$ in the obvious way \cite{Fr21}.

\begin{definition} [Friedman \cite{Fr21}]\label{ETF} The axioms of elementary theory of functions (ETF) consist of the following $L_{fcn}$-formulas:
\begin{tabbing}
\hspace{1em} \= (1) Successor axioms: \\

\hspace{2em} \= i. $S(n)\not=0$ \\

\hspace{2em} \= ii. $S(n)=S(m)\rightarrow n=m$ \\

\hspace{2em} \= iii. $n\not=0\rightarrow(\exists m)(S(m)=n)$ \\

\hspace{1em} \= (2) Initial functions axioms: \\

\hspace{2em} \= i. $(\exists f)(\forall m)(f(m)=n)$ (constant unary functions)\\

\hspace{2em} \= ii. $(\exists f)(\forall m,n,r)(f(m,n,r)=m)$ (projection function)\\

\hspace{2em} \= iii. $(\exists f)(\forall m,n,r)(f(m,n,r)=n)$ (projection function)\\

\hspace{2em} \= iv. $(\exists f)(\forall m,n,r)(f(m,n,r)=r)$ (projection function)\\

\hspace{2em} \= v. $(\exists f)(\forall n)(f(n)=S(n))$ ($S(n)$ defines a unary function) \\

\hspace{1em} \= (3) Composition axioms: \\

\hspace{2em} \= i. $(\exists f)(\forall m,n,r)(f(m,n,r)=g(m,n))$ \\

\hspace{2em} \= ii. $(\exists f)(\forall m,n,r)(f(m,n,r)=g(m))$ \\

\hspace{2em} \= iii. $(\exists f)(\forall m,n)(f(m,n)=g(m,n,r))$ \\

\hspace{2em} \= iv. $(\exists f)(\forall n)(f(m)=g(m,n,r))$ \\

\hspace{2em} \= v. $(\exists f)(\forall m,n,r)(f(m,n,r)=g(h_1(m,n,r), h_2(m,n,r), h_3(m,n,r)))$ \\

\hspace{1em} \= (4) Primitive recursion axioms ($PRA$): \\

\hspace{2em} \= $(\exists f)(\forall m)(f(m,0)=g(m)\&(\forall n)(f(m,S(n))=h(m,n,f(m,n))))$ \\

\hspace{1em} \= (5) Permutation axiom (PERM): \\

\hspace{2em} \= $(\forall n)(\exists!m)(f(m)=n)\rightarrow(\exists g)(\forall n)(f(g(n))=n)$. \\

\hspace{1em} \= (6) Rudimentary induction axiom: \\

\hspace{2em} \= $f(0)=g(0)\&(\forall n)(f(n)=g(n)\rightarrow f(S(n))=g(S(n)))\rightarrow f(n)=g(n)$.
\end{tabbing}
\end{definition}

In \cite{Fr22} Friedman formulated tasks of studying subsystems of ETF, weakening the primitive recursion axiom and finding reversals to ETF. In the present paper we address these tasks. We focus on studying subsystems of $ETF$. Successor axioms (1) and initial function axioms (2) from Definiton \ref{ETF} are included in all considered subsystems. Adding composition axioms (3) we get the first object of our study that we designated $COM_{fcn}$. Adding induction axiom (6) to $COM_{fcn}$ we get the subsystem designated $COMI_{fcn}$. The subsystem consisting of axioms (1)-(4) and (6) is designated as $PRA_{fcn}$. So, the system $PRA_{fcn}$ is $COMI_{fcn}+PRA$, and the system $ETF$ is $PRA_{fcn}+PERM$.

We study the following statements.

\begin{definition}\label{statements}

\

$WPRA:(\exists f)(\forall m)(f(m,0)=g(m)\&(\forall n)(f(m,S(n))=h(m,S(n),f(m,n))))$

$MIN^1:(\forall m)(\exists! n)(f(m,n)=0)\rightarrow(\exists g)(\forall m)(f(m,g(m))=0)$.

$MIN^2:(\forall m)(\exists n)(f(m,n)=0)\rightarrow(\exists g)(\forall m)(g(m)=(\mu n)(f(m,n)=0))$.

$MIN^3:(\forall m,n)(\exists r)(f(m,n,r)=0)\rightarrow(\exists g)(\forall m,n)(g(m,n)=(\mu r)(f(m,n,r)=0))$.
\end{definition}

We prove that inductions on any variable for unary, binary and ternary functions are pairwise equivalent over $COM_{fcn}$.

We show that weakened primitive recursion axiom $WPRA$ is equivalent to $PRA$ over $COMI_{fcn}$, and hence $PRA_{fcn}$ is equivalent to $COMI_{fcn}+WPRA$.

We also prove that minimization axioms $MIN^1$, $MIN^2$, $MIN^3$ are equivalent to each other and to $PERM$ over $PRA_{fcn}$ (complementing the result of Friedman \cite{Fr21} that $PRA_{fcn}\vdash PERM\rightarrow MIN^1\rightarrow MIN^2\rightarrow MIN^3$), and hence $ETF$ is equivalent to $PRA_{fcn}+MIN^k$ and to $COMI_{fcn}+WPRA+MIN^k$ for $k=1,2,3$.

To make the paper self-contained and to make strict reverse mathematics more accessible to general mathematics community we have included proofs of those lemmas from the unpublished Friedman's preprint \cite{Fr21} that we use.

\section{The system $COMI_{fcn}$}

We start from the subsystem $COM_{fcn}$ of ETF consisting of successor axioms, initial functions axioms and composition axioms. Lemma \ref{lemma1} shows that this subsystem carries the basic ``logic'' content of ETF, while induction, primitive recursion, permutation and other axioms carries additional ``mathematical'' content. In particular, it is the composition axioms that allow to substitute terms and construct new formulas, and this, in fact, is quite evident, since mathematical composition of functions has the same nature as logical substitution of terms (in older books and papers 'composition' of functions was even called 'substitution' of functions, see, e.g., the classical Kleene's book \cite{Kl52}).

This humble subsystem allows to prove equivalence between induction axioms for unary, binary and ternary functions as Lemma \ref{lemma2} shows. However enriching this subsystem with rudimentary induction axiom leads to more powerful subsystem $COMI_{fcn}$ that looks like as a suitable base theory to compare logical power of subclasses of computable and primitive recursive functions (for instance, such as levels of the  Grzegorzcyk hierarchy \cite{Grze}). To demonstrate this we prove that $WPRA$ and $PRA$ are equivalent over $COMI_{fcn}$.

\begin{definition} The axioms of $COMI_{fcn}$ consist of the following list of $L_{fcn}$-formulas, the axioms of $COM_{fcn}$ consist of the formulas (1)-(3) from the list:
\begin{tabbing}
\hspace{1em} \= (1) Successor axioms: \\

\hspace{2em} \= i. $S(n)\not=0$ \\

\hspace{2em} \= ii. $S(n)=S(m)\rightarrow n=m$ \\

\hspace{2em} \= iii. $n\not=0\rightarrow(\exists m)(S(m)=n)$ \\

\hspace{1em} \= (2) Initial functions axioms: \\

\hspace{2em} \= i. $(\exists f)(\forall m)(f(m)=n)$ (constant unary functions)\\

\hspace{2em} \= ii. $(\exists f)(\forall m,n,r)(f(m,n,r)=m)$ (projection function)\\

\hspace{2em} \= iii. $(\exists f)(\forall m,n,r)(f(m,n,r)=n)$ (projection function)\\

\hspace{2em} \= iv. $(\exists f)(\forall m,n,r)(f(m,n,r)=r)$ (projection function)\\

\hspace{2em} \= v. $(\exists f)(\forall n)(f(n)=S(n))$ ($S(n)$ defines a unary function) \\

\hspace{1em} \= (3) Composition axioms: \\

\hspace{2em} \= i. $(\exists f)(\forall m,n,r)(f(m,n,r)=g(m,n))$ \\

\hspace{2em} \= ii. $(\exists f)(\forall m,n,r)(f(m,n,r)=g(m))$ \\

\hspace{2em} \= iii. $(\exists f)(\forall m,n)(f(m,n)=g(m,n,r))$ \\

\hspace{2em} \= iv. $(\exists f)(\forall n)(f(m)=g(m,n,r))$ \\

\hspace{2em} \= v. $(\exists f)(\forall m,n,r)(f(m,n,r)=g(h_1(m,n,r), h_2(m,n,r), h_3(m,n,r)))$ \\

\hspace{1em} \= (6) Rudimentary induction axiom: \\

\hspace{2em} \= $f(0)=g(0)\&(\forall n)(f(n)=g(n)\rightarrow f(S(n))=g(S(n)))\rightarrow f(n)=g(n)$.
\end{tabbing}
\end{definition}

First note, that composition axioms $(i)$-$(ii)$ allow to convert binary and unary functions to ternary functions with dummy arguments, and composition axioms $(iii)$-$(iv)$ allow us to transfer from ternary functions to binary and unary functions. These axioms made it possible to significantly simplify the list of axioms for $COM_{fcn}$. For instance, we don't need axioms for ternary and binary constant functions, since ternary constant functions are obtained from unary constant functions and composition axiom $(ii)$, and binary constant functions are obtained from ternary constant functions and composition axiom $(iii)$. The same is true for binary and unary projection functions that are obtained from ternary projection functions and composition axioms $(iii)$-$(iv)$.

Composition axiom $(v)$ together with projection function axioms allows to change function arguments places. For instance, set $h_1(k,m,n)=m$ and $h_2(k,m,n)=k$ and for any given $g(k,m,n)$ applying composition axiom $(v)$ we have $f(k,m,n)=g(h_1(k,m,n),h_2(k,m,n),n)=g(m,k,n)$.

All of the above is summarized in the following lemma that was proved in \cite{Fr21} for $ETF\backslash PERM$, but in fact is true for $COM_{fcn}$.

\begin{lemma} [Friedman \cite{Fr21}]\label{lemma1} Let $m,n,r$ be distinct variables and $t$ be a term, then the following are provable in $COM_{fcn}$.

(i) $(\exists f)(\forall m,n,r)(f(m,n,r)=t)$

(ii) $(\exists f)(\forall m,n)(f(m,n)=t)$

(iii) $(\exists f)(\forall n)(f(n)=t)$.
\end{lemma}
\begin{proof}

\

We prove $(i)$ by external induction on $t$.

If $t$ is $0$ or a variable other than $m,n,r$, then we use ternary constant functions obtained from initial functions axioms $(i)$ and composition axiom $(ii)$. If $t$ is among $m,n,r$, then use projection functions from initial functions axioms $(ii)-(iv)$.

The induction step. If $t$ is $g(s_1,s_2,s_3)$, then by induction hypothesis, let $h_1(m,n,r)=s_1, h_2(m,n,r)=s_2, h_3(m,n,r)=s_3$ hold for all $m,n,r$, so by composition axiom $(v)$ we have $(\exists f)(\forall m,n,r) (f(m,n,r)=g(h_1(m,n,r), h_2(m,n,r), h_3(m,n,r)))$. If $t$ is $g(s_1,s_2)$, then we convert it to a ternary function $g'$ by composition axiom $(i)$ and apply composition axiom $(v)$. If $t$ is $g(s_1)$, then we convert it to a ternary function $g'$ by composition axiom $(ii)$ and apply composition axiom $(v)$. If $t$ is $S(s_1)$, we convert function symbol $S$ to a unary function $g$ by initial functions axiom $(v)$ and apply composition axioms $(ii)$ and $(v)$.

We can derive $(ii)$ and $(iii)$ from $(i)$ and composition axioms $(iii)$ and $(iv)$ respectively.
\end{proof}

The next lemma shows that composition axioms are powerful enough to switch between inductions on any variable for unary, binary and ternary functions.

\begin{lemma} \label{lemma2} The following assertions are pairwise equivalent over $COM_{fcn}$.

(i) $f(0)=g(0)\&(\forall n)(f(n)=g(n)\rightarrow f(S(n))=g(S(n)))\rightarrow f(n)=g(n)$

(ii) $f(m,0)=g(m,0)\&(\forall n)(f(m,n)=g(m,n)\rightarrow f(m,S(n))=g(m,S(n)))\rightarrow f(m,n)=g(m,n)$

(iii) $f(0,n)=g(0,n)\&(\forall m)(f(m,n)=g(m,n)\rightarrow f(S(m),n)=g(S(m),n))\rightarrow f(m,n)=g(m,n)$

(iv) $f(k,m,0)=g(k,m,0)\&(\forall n)(f(k,m,n)=g(k,m,n)\rightarrow f(k,m,S(n))=g(k,m,S(n)))\rightarrow f(k,m,n)=g(k,m,n)$

(v) $f(k,0,n)=g(k,0,n)\&(\forall m)(f(k,m,n)=g(k,m,n)\rightarrow f(k,S(m),n)=g(k,S(m),n))\rightarrow f(k,m,n)=g(k,m,n)$

(vi) $f(0,m,n)=g(0,m,n)\&(\forall k)(f(k,m,n)=g(k,m,n)\rightarrow f(S(k),m,n)=$
$g(S(k),m,n))\rightarrow f(k,m,n)=g(k,m,n)$
\end{lemma}
\begin{proof}
Each equivalence is proved in the same way. For example, here is the proof $(ii)\Rightarrow(vi)$. Let we have $f(0,m,n)=g(0,m,n)\&(\forall k)(f(k,m,n)=g(k,m,n)\rightarrow f(S(k),m,n)=g(S(k),m,n))$. By Lemma \ref{lemma1}.(ii) for any $n$ there are $h_1$ and $h_2$ such that for all $m,k$ we have $h_1(m,k)=f(k,m,n)$, $h_2(m,k)=g(k,m,n)$. This means that $h_1(m,0)=h_2(m,0)\&(\forall k)(h_1(m,k)=h_2(m,k)\rightarrow h_1(m,S(k))=h_2(m,S(k)))$. By $(ii)$ we have $h_1(m,k)=h_2(m,k)$, i.e. for any $n$ we have $f(k,m,n)=g(k,m,n)$.
\end{proof}

For $COMI_{fcn}$ the Lemma \ref{lemma2} can be reformulated in the following general form that was stated for $ETF\backslash PERM$ and called Equational Induction in \cite{Fr21}.

\begin{lemma} [Friedman \cite{Fr21}] \label{lemma3} Let $s,t$ be terms, then $COMI_{fcn}$ proves induction for the equation $s=t$ on any variable $n$ of sort $\omega$.
\end{lemma}
\begin{proof} By Lemma \ref{lemma1} let $f(n)=s$ and $g(n)=t$, where $f$ and $g$ depend internally on the functions in $s,t$ and the variables over $\omega$ in $s,t$ other than $n$. Then apply rudimentary induction axiom to $f(n)=g(n)$.
\end{proof}

The following notation will be used to point to arbitrary external elements of $\omega$ and to specific elements of $\omega$ such as $1$ and $2$.

\begin{definition} [Friedman \cite{Fr21}]\label{definition*}
For each $n$, we will denote by $n^*$ the term $S(S(...S(0)))$, where $S$ is taken $n$ times. We also will simply use the notation $1$ and $2$ instead of $S(0)$, $S(S(0))$.
\end{definition}

Our next steps are to define basic arithmetic functions and then prove that they satisfy usual arithmetic properties. In \cite{Fr21} this was done by primitive recursion axioms. We add $WPRA:(\exists f)(\forall m)(f(m,0)=g(m)\&(\forall n)(f(m,S(n))=h(m,S(n),f(m,n))))$ to $COMI_{fcn}$ for this purpose.

Some arithmetic functions can be defined in $COMI_{fcn}+WPRA$ in a usual way, however defining the predecessor function $P$ becomes a little tricky. The main difference between $WPRA$ and $PRA$ is that predecessor is already implicitly present in $PRA$, so the primitive recursive function $f(m,S(n))$ immediately gets access to $n$. In the case of $WPRA$ the weakened primitive recursive function $f(m,S(n))$ gets access only to $S(n)$. So we are forced to define the predecessor function $P$ using some combinatorics, and then prove that such a function satisfies the main characteristic predecessor property $P(S(n))=n$.

First, show that $WPRA$ can be used to obtain unary iteration functions.

\begin{lemma}
The following is provable in $COMI_{fcn}+WPRA$.

$ITER:(\exists f)(f(0)=r\wedge(\forall n)(f(S(n))=h(f(n))))$.
\end{lemma}
\begin{proof} Let $r$ and $h(n)$ be given. By Lemma \ref{lemma1} set $h'(m,r,n)=h(n)$ and apply $WPRA$ to $h'$ and the constant function $g(n)=r$, thus obtaining function $f'(m,0)=r\wedge(\forall n)(f'(m,S(n))=h'(m,S(n),f'(m,n)))$. Then by Lemma \ref{lemma1} set $f(n)=f'(m,n)$. We have $f(0)=r\wedge(\forall n)(f(S(n))=f'(m,S(n))=h'(m,S(n),f'(m,n))=h(f'(m,n))=h(f(n)))$.
\end{proof}

Now we prove the existence of arithmetic functions. Note, that we don't need to expand the language $L_{fcn}$ with symbols from the following lemma, since $L_{fcn}$ is a many-sorted language with function variables. We just agree to use these symbols for convenience. For instance, using $WPRA$ we set for all $m$, $f(m,0)=m\wedge(\forall n)(f(m,S(n))=S(f(m,n)))$ and then abbreviate by ``$m+n$'' in a formula $\varphi$ that ``$\varphi$ holds of some function $f$ such that $f(m,0)=n\wedge(\forall n)(f(m,S(n))=S(f(m,n)))$''.

\begin{lemma} $COMI_{fcn}+WPRA$\label{lemmaarithm} proves the existence of functions $+$,$\cdot$, $sg$,$\overline{sg}$, $odd$, $f'$ (an auxiliary function), $P$,$-$, such that the following holds with the variables of sort $\omega$:

(i) $m+0=m\wedge(\forall n)(m+S(n)=S(m+n))$

(ii) $m\cdot0=0\wedge(\forall n)(m\cdot S(n)=m\cdot n+m)$

(iii) $sg(0)=0\wedge(\forall n)(sg(S(n))=1)$

(iv) $\overline{sg}(0)=1\wedge(\forall n)(\overline{sg}(S(n))=0)$

(v) $odd(0)=0\wedge(\forall n)(odd(S(n))=\overline{sg}(odd(n)))$

(vi) $f'(0)=0\wedge(\forall n)(f'(S(n)))=\overline{sg}(f'(n))+f'(n)\cdot S(f'(n))$

(vii) $P(0)=0\wedge(\forall n)(P(S(n)))=\overline{sg}(odd(f'(S(n))))\cdot S(P(n))$

(viii) $m-0=m\wedge(\forall n)(m-S(n)=P(m-n))$.

\end{lemma}
\begin{proof}
Each of functions is defined using $WPRA$, $ITER$ and composition axioms with previously defined functions.
\end{proof}

It follows from Lemma \ref{lemma3} that each of these functions is extensionally unique in the sense that any two such functions agree everywhere.

The next theorem shows that these functions satisfy main properties of corresponding arithmetic functions. Proofs of $(i)-(vii)$ and $(xiii)-(xiv)$ are from \cite{Fr21}.

\begin{theorem}\label{Theorem1} The following are provable in $COMI_{fcn}+WPRA$.

(i) $n+1=S(n)$

(ii) $0+n=n$

(iii) $S(n)+m=n+S(m)$

(iv) $n+m=m+n$

(v) $(n+m)+r=n+(m+r)$

(vi) $n+m=0\leftrightarrow n=0\wedge m=0$

(vii) $0\cdot n=0$

(viii) $S(n)\cdot m=n\cdot m+m$

(ix) $m\cdot n=n\cdot m$

(x) $1\cdot n=n$

(xi) $m\cdot n=0\leftrightarrow n=0\vee m=0$

(xii) $n+n=2\cdot n$.

(xiii) $sg(0)=0$ and $sg(n)=1$ for all $n\not=0$

(xiv) $\overline{sg}(0)=1$ and $\overline{sg}(n)=0$ for all $n\not=0$

(xv) $\overline{sg}(n\cdot S(n))=\overline{sg}(n)$

(xvi) $odd(1)=1$

(xvii) $f'(1)=1$

(xviii) $\overline{sg}(f'(0))=1$ and $\overline{sg}(f'(n))=0$ for all $n\not=0$

(xix) $odd(2\cdot n)=0$

(xx) $odd(S(S(n)))=odd(n)$

(xxi) $odd(m+2\cdot n)=odd(m)$

(xxii) $odd(n+S(n))=1$

(xxiii) $odd(n\cdot S(n))=0$

(xxiv) $odd(f'(1))=1$ and $odd(f'(n))=0$ for all $n\not=1$

(xxv) $P(0)=0$ and $S(P(n))=n$ for all $n\not=0$

(xxvi) $P(S(n))=n$
\end{theorem}
\begin{proof}

\

$(i)$ $n+1=n+S(0)=S(n+0)=S(n)$ by Lemma \ref{lemmaarithm}.(i).

$(ii)$ Use Equational Induction from Lemma \ref{lemma2} on $n$. If $n=0$ then $0+0=0$ by Lemma \ref{lemmaarithm}.(i). Suppose $0+n=n$, then $0+S(n)=S(0+n)=S(n)$.

$(iii)$ Use Equational Induction on $m$. If $m=0$ then $S(n)+0=S(n)=S(n+0)=n+S(0)$ by Lemma \ref{lemmaarithm}.(i). Suppose, $S(n)+m=n+S(m)$, then $S(n)+S(m)=S(S(n)+m)=S(n+S(m))=n+S(S(m))$ by Lemma \ref{lemmaarithm}.(i).

$(iv)$ Use Equational Induction on $m$. If $m=0$ then $n+0=n=0+n$ by $(ii)$ and Lemma \ref{lemmaarithm}.(i). Suppose $n+m=m+n$, then $n+S(m)=S(n+m)=S(m+n)=m+S(n)=S(m)+n$ by $(iii)$ and Lemma \ref{lemmaarithm}.(i).

$(v)$ Use Equational Induction on $r$. If $r=0$ then $(n+m)+0=n+m=n+(m+0)$ by Lemma \ref{lemmaarithm}.(i). Suppose $(n+m)+r=n+(m+r)$, then $(n+m)+S(r)=S((n+m)+r)=S(n+(m+r))=n+S(m+r)=n+(m+S(r))$ by Lemma \ref{lemmaarithm}.(i).

$(vi)$ If $n=0\wedge m=0$, then $n+m=0$ by Lemma \ref{lemmaarithm}.(i). Let $n+m=0$ and suppose that $m\not=0$. Then by successor axiom $(iii)$, $(\exists r)(S(r)=m)$ and $n+m=n+S(r)=S(n+r)\not=0$ by successor axiom $(i)$, contradiction. If $n\not=0$, then by successor axiom $(iii)$, $(\exists t)(S(t)=m)$ and $n+m=m+n=m+S(t)=S(m+t)$ by $(iv)$ and successor axiom $(i)$, contradiction.

$(vii)$ Use Equational Induction on $n$. If $n=0$, then $0\cdot0=0$ by Lemma \ref{lemmaarithm}.(ii). Suppose $0\cdot n=0$, then $0\cdot S(n)=0\cdot n+0=0+0=0$ by Lemma \ref{lemmaarithm}.(ii) and Lemma \ref{lemmaarithm}.(i).

$(viii)$ Use Equational Induction on $m$. If $m=0$ then $S(n)\cdot 0=0=0+0=n\cdot 0+0$ by Lemma \ref{lemmaarithm}.(i) and Lemma \ref{lemmaarithm}.(ii). Suppose $S(n)\cdot m=n\cdot m+m$, then $S(n)\cdot S(m)=S(n)\cdot m+S(n)=n\cdot m+m+S(n)=n\cdot m+S(n)+m=n\cdot m+n+1+m=n\cdot S(m)+m+1=n\cdot S(m)+S(m)$ by $(i)$, $(iv)$ and Lemma \ref{lemmaarithm}.(ii).

$(ix)$ Use Equational Induction on $n$. If $n=0$, then $m\cdot0=0=0\cdot m$ by $(vii)$ and Lemma \ref{lemmaarithm}.(ii). Suppose $m\cdot n=n\cdot m$, then $m\cdot S(n)=m\cdot n+m=n\cdot m+m=S(n)\cdot m$ by $(viii)$ and Lemma \ref{lemmaarithm}.(ii).

$(x)$ Use Equational Induction on $n$. If $n=0$, then $1\cdot0=0$ by Lemma \ref{lemmaarithm}.(ii). Suppose $1\cdot n=n$, then $1\cdot S(n)=1\cdot n+1=n+1=S(n)$ by $(i)$, Lemma \ref{lemmaarithm}.(ii) and Lemma \ref{lemmaarithm}.(i).

$(xi)$ If $n=0\vee m=0$, then $m\cdot n=0$ by $(vii)$ and Lemma \ref{lemmaarithm}.(ii). Let $m\cdot n=0$ and suppose that $m\not=0\wedge n\not=0$. Then by successor axiom $(iii)$, $(\exists r)(S(r)=m)$ and $(\exists t)(S(t)=n)$, thus by Lemma \ref{lemmaarithm}.(i) and Lemma \ref{lemmaarithm}.(ii), $m\cdot n=S(r)\cdot S(t)=S(r)\cdot t+S(r)=S(S(r)\cdot t+r)\not=0$ by successor axiom $(i)$, contradiction.

$(xii)$ Use Equational Induction on $n$. If $n=0$ then $0+0=0=2\cdot 0$ by Lemma \ref{lemmaarithm}.(i) and Lemma \ref{lemmaarithm}.(ii). Suppose $n+n=2\cdot n$, then $S(n)+S(n)=S(S(n)+n)=S(n+S(n))=S(S(n+n))=S(S(2\cdot n))=S(S(2\cdot n)+0)=S(2\cdot n+S(0))=2\cdot n+S(S(0))=2\cdot n+2=2\cdot S(n)$ by $(iv)$, Lemma \ref{lemmaarithm}.(i) and Lemma \ref{lemmaarithm}.(ii).

$(xiii)$ If $n=0$ then $sg(n)=0$ by Lemma \ref{lemmaarithm}.(iii). If $n\not=0$, then by successor axiom $(iii)$, $(\exists m)(S(m)=n)$ and $sg(n)=sg(S(m))=1$ by Lemma \ref{lemmaarithm}.(iii).

$(xiv)$ If $n=0$ then $\overline{sg}(0)=1$ by Lemma \ref{lemmaarithm}.(iv). If $n\not=0$, then by successor axiom $(iii)$, $(\exists m)(S(m)=n)$, and $\overline{sg}(n)=\overline{sg}(S(m))=0$ by Lemma \ref{lemmaarithm}.(iv).

$(xv)$ Use Equational Induction on $n$. If $n=0$ then by $(vii)$, $\overline{sg}(0\cdot S(0))=\overline{sg}(0)$. The induction step. Since $S(n)\not=0$ and $S(S(n))\not=0$, we have $S(n)\cdot S(S(n))\not=0$ by $(xi)$, thus by $(xiv)$, $\overline{sg}(S(n)\cdot S(S(n))))=0=\overline{sg}(S(n))$.

$(xvi)$ By Lemma \ref{lemmaarithm}.(iv) and Lemma \ref{lemmaarithm}.(v), $odd(1)=odd(S(0))=\overline{sg}(odd(0))=\overline{sg}(0)=1$.

$(xvii)$ By $(vii)$, Lemma \ref{lemmaarithm}.(vi) and Lemma \ref{lemmaarithm}.(iv), $f'(1)=\overline{sg}(f'(0))+f'(0)\cdot S(f'(0))=\overline{sg}(0)+0\cdot S(0)=1+0=1$.

$(xviii)$ By Lemma \ref{lemmaarithm}.(v) and Lemma \ref{lemmaarithm}.(vi), $\overline{sg}(f'(0))=\overline{sg}(0)=1$. Use Equational Induction on $n$ to show that $\overline{sg}(f'(S(n)))=0$ for all $n$. If $n=0$ then by $(xvii)$, $\overline{sg}(f'(S(0)))=\overline{sg}(1)=0$. Suppose $\overline{sg}(f'(S(n)))=0$, then by $(xv)$ and Lemma \ref{lemmaarithm}.(vi), $\overline{sg}(f'(S(S(n))))=\overline{sg}(\overline{sg}(f'(S(n)))+f'(S(n))\cdot S(f'(S(n))))=\overline{sg}(0+f'(S(n))\cdot S(f'(S(n))))=\overline{sg}(f'(S(n)))=0$.

$(xix)$ Use Equational Induction on $n$. If $n=0$ then by Lemma \ref{lemmaarithm}.(ii) and Lemma \ref{lemmaarithm}.(v), $odd(2\cdot 0)=odd(0)=0$. Suppose $odd(2\cdot n)=0$, then $odd(2\cdot S(n)))=odd(2\cdot n+2)=odd(2\cdot n+S(S(0)))=odd(S(2\cdot n+S(0)))=odd(S(S(2\cdot n+0)))=\overline{sg}(odd(S(2\cdot n)))=\overline{sg}(\overline{sg}(odd(2\cdot n)))=\overline{sg}(\overline{sg}(0))=\overline{sg}(1)=0$ by Lemma \ref{lemmaarithm}.(i), \ref{lemmaarithm}.(ii), \ref{lemmaarithm}.(iv), \ref{lemmaarithm}.(v).

$(xx)$ Use Equational Induction on $n$. If $n=0$ then $odd(S(S(0)))=\overline{sg}(\overline{sg}(odd(0)))=\overline{sg}(\overline{sg}(0))=\overline{sg}(1)=0=odd(0)$. Suppose $odd(S(S(n)))=odd(n)$, then $odd(S(S(S(n))))=\overline{sg}(odd(S(S(n))))=\overline{sg}(odd(n))=odd(S(n))$.

$(xxi)$ Use Equational Induction on $n$. If $n=0$ then by Lemma \ref{lemmaarithm}.(i), Lemma \ref{lemmaarithm}.(ii), $odd (m+2\cdot 0)=odd(m+0)=odd(m)$. Suppose $odd(m+2\cdot n)=odd(m)$, then by $(i)$, $(iv)$, $(v)$, $(xii)$, $(xx)$ and Lemma \ref{lemmaarithm}.(i), $odd(m+2\cdot S(n))=odd(m+S(n)+S(n))=odd(m+n+1+n+1)=odd(m+n+n+S(0)+S(0))=odd(S(m+2\cdot n+S(0)+0))=\overline{sg}(odd(m+2\cdot n+S(0)))=\overline{sg}(odd(S(m+2\cdot n+0)))=\overline{sg}(\overline{sg}(odd(m+2\cdot n)))=\overline{sg}(\overline{sg}(odd(m)))=\overline{sg}(odd(S(m)))=odd(S(S(m)))=odd(m)$.

$(xxii)$ Use Equational Induction on $n$. If $n=0$ then by $(xvi)$ and Lemma \ref{lemmaarithm}.(v), $odd(0+S(0))=odd(1)=1$. Suppose $odd(n+S(n))=1$, then $odd(S(n)+S(S(n)))=odd(S(S(n)+S(n)))=\overline{sg}(odd(S(n)+S(n)))=\overline{sg}(odd(S(S(n)+n)))=\overline{sg}(\overline{sg}(odd(S(n)+n)))=
\overline{sg}(\overline{sg}(odd(n+S(n))))=\overline{sg}(\overline{sg}(1))=\overline{sg}(0)=1$ by $(iv)$, Lemma \ref{lemmaarithm}.(i) and Lemma \ref{lemmaarithm}.(v).

$(xxiii)$ Use Equational Induction on $n$. If $n=0$ then by $(vii)$ and Lemma \ref{lemmaarithm}.(v) $odd(0\cdot S(0))=odd(0)=0$. Suppose $odd(n\cdot S(n))=0$, then $odd(S(n)\cdot S(S(n)))=odd(S(n)\cdot S(n)+S(n)))=odd(S(S(n)\cdot S(n)+n))=\overline{sg}(odd(S(n)\cdot S(n)+n))=\overline{sg}(odd(S(n)\cdot n+S(n)+n))=\overline{sg}(odd(S(n)\cdot n+n+S(n)))=\overline{sg}(odd(S(S(n)\cdot n+n+n)))=\overline{sg}(\overline{sg}(odd(S(n)\cdot n+2\cdot n)))=\overline{sg}(\overline{sg}(odd(n\cdot S(n)+2\cdot n)))=\overline{sg}(\overline{sg}(odd(n\cdot S(n))))=\overline{sg}(\overline{sg}(0))=0$ by $(iv)$, $(ix)$, $(xix)$, $(xxi)$, Lemma \ref{lemmaarithm}.(i), Lemma \ref{lemmaarithm}.(ii) and Lemma \ref{lemmaarithm}.(v).

$(xxiv)$ By Lemma \ref{lemmaarithm}.(v) and Lemma \ref{lemmaarithm}.(vi), $odd(f'(0))=odd(0)=0$. By $(xvi)$, $(xvii)$ and Lemma \ref{lemmaarithm}.(vi), $odd(f'(1))=odd(1)=1$.

Now, use Equational Induction on $n$ to show that $odd(f'(S(S(n))))=0$ for all $n$. If $n=0$ then $odd(f'(S(S(0))))=odd(\overline{sg}(f'(S(0)))+f'(S(0))\cdot S(f'(S(0))))=odd(\overline{sg}(1)+1\cdot S(1))=odd(0+S(1))=odd(S(1))=\overline{sg}(odd(1)))=\overline{sg}(1)=0$ by $(x)$, $(xiv)$, $(xvi)$, $(xvii)$, Lemma \ref{lemmaarithm}.(v) and Lemma \ref{lemmaarithm}.(vi).

The induction step. Suppose $odd(f'(S(S(n))))=0$, then $odd(f'(S(S(S(n)))))=odd(\overline{sg}(f'(S(S(n))))+f'(S(S(n)))\cdot S(f'(S(S(n)))))=odd(0+f'(S(S(n))))\cdot S(f'(S(S(n)))))=0$ by $(xviii)$ and $(xxiii)$.

$(xxv)$ By Lemma \ref{lemmaarithm}.(vii), $P(0)=0$.

Use Equational Induction on $n$ to show that $S(P(S(n)))=S(n)$ for all $n$. By $(xxiv)$ and Lemma \ref{lemmaarithm}.(vii), $P(1)=\overline{sg}(odd(f'(1)))\cdot S(P(0))=\overline{sg}(1)\cdot S(P(0))=0\cdot 1=0$, so $S(P(S(0)))=S(P(1))=S(0)$.

Suppose $S(P(S(n)))=S(n)$, then by $(xxiv)$ and Lemma \ref{lemmaarithm}.(vii), $S(P(S(S(n))))=S(\overline{sg}(odd(f'(S(S(n)))))\cdot S(P(S(n))))=S(\overline{sg}(0)\cdot S(n))=S(1\cdot S(n))=S(S(n))$.

$(xxvi)$ Let $S(n)=m$. By $(xxv)$, $S(P(m))=m$, so $S(n)=m=S(P(m))$, hence $n=P(m)$ by successor axiom $(ii)$. So we have $P(S(n))=P(m)=n$.
\end{proof}

\begin{corollary}
$PRA$ and $WPRA$ are pairwise equivalent over $COMI_{fcn}$.
\end{corollary}
\begin{proof}

\

Assume $PRA$. Let $g(m)$ and $h(m,n,r)$ be given. Set $h'(m,n,r)=h(m,S(n),r)$ and apply $PRA$ to get the function $f$ such that $(\forall m)(f(m,0)=g(m)\&(\forall n)(f(m,S(n))=h'(m,n,f(m,n))))=h(m,S(n),f(m,n))$.

Assume $WPRA$. Let $g(m)$ and $h(m,n,r)$ be given. Set $h'(m,n,r)=h(m,P(n),r)$ and apply $WPRA$ to get the function $f$ such that $(\forall m)(f(m,0)=g(m)\&(\forall n)(f(m,S(n))=h'(m,S(n),f(m,n))))=h(m,P(S(n)),f(m,n))=h(m,n,f(m,n))$.
\end{proof}

\section{The system $PRA_{fcn}$}

\

This section contains several results from the unpublished Friedman's preprint \cite{Fr21} showing that primitive recursion axioms added to $COMI_{fcn}$ allow to extend induction to broader classes of formulas, to define functions by conditional terms, to define $<$ and the Cantor pairing function. These results could be just cited, however for the convenience of the reader we fill in small gaps and present their full proofs in standard notations. For instance, we prove Lemma \ref{quotient} that shows the existence of the quotient function that gives the result of the division of a number by 2. This function is necessary for the Friedman's proof of $PRA_{fcn}\vdash PERM\rightarrow MIN^1$. We hope that such an exposition of proofs from \cite{Fr21} is not superfluous and will help the present paper to become self-sufficient as well as help strict reverse mathematics to become more accessible.

The main result of this section is Theorem \ref{Theorem2} showing that minimization axioms $MIN^1$, $MIN^2$, $MIN^3$ are equivalent to each other and to $PERM$ over $PRA_{fcn}$.

As was shown in the previous section, rudimentary induction axioms together with composition axioms gives Equational Induction, i.e. induction for equations $s=t$ where $s$ and $t$ are terms. By adding primitive recursion axioms one can prove induction for propositional combinations of equations as Lemma \ref{induction} shows.

\begin{definition}[Friedman \cite{Fr21}] An open formula is propositional combination of equations. A singular open formula is a propositional combination of equations of the form $t=n^*$.
\end{definition}

\begin{lemma}[Friedman \cite{Fr21}]\label{induction} The following are provable in $PRA_{fcn}$.

(i) $1-S(n)=0$

(ii) $2-S(S(n))=0$

(iii) $n=0\vee n=1\vee 2-n=0$

(iv) If $m\not=0$ then $P(n)=m\leftrightarrow n=S(m)$

(v) $0-n=0$

(vi) $n-1=0\leftrightarrow n=0\vee n=1$

(vii) $1-n=0\leftrightarrow n\not=0$

(viii) $n=1\leftrightarrow (n-1)+(1-n)=0$

(ix) $(2-n)\cdot n$ is $1$ at $1$ and $0$ elsewhere

(x) For $m\geq 2$, $n=m^*\leftrightarrow P(...P(n))=1$, where there are $m-1$ $P$'s

(xi) Every singular open formula is provably equivalent to a propositional combination of equations of the form $t=1$

(xii) Every singular open formula is provably equivalent to an equation $t=0$

(xiii) Induction for singular open formulas on any variable $n$ of sort $\omega$

(xiv) $P(n)=0\rightarrow n=0\vee n=1$

(xv) $S(n)-S(m)=n-m$

(xvi) $(n+m)-m=n$

(xvii) $n+r=m+r\rightarrow n=m$

(xviii) $n=m\vee r=0\leftrightarrow n\cdot sg(r)=m\cdot sg(r)$

(xix) $n-n=0$

(xx) $S(n)-n=1$

(xxi) If $n\not=0$ then $P(n)+S(m)=n+m$

(xxii) $n-m=0\vee (n-m)+m=n$

(xxiii) $n-m=m-n=0\rightarrow n=m$

(xxiv) $n-m=m-n=0\leftrightarrow n=m$

(xxv) $n-m=0\vee m-n=0$

(xxvi) $n-m=0\leftrightarrow(\exists r)(n+r=m)$

(xxvii) Every open formula is provably equivalent to an equation $t=0$

(xxviii) Induction for open formulas on any variable $n$ of sort $\omega$
\end{lemma}
\begin{proof}

\

$(i)$ Use Equational Induction on $n$. If $n=0$ then by Lemma \ref{lemmaarithm}.(viii) and Theorem \ref{Theorem1}.(xxvi), $1-S(0)=P(1-0)=P(1)=P(S(0))=0$. Suppose $1-S(n)=0$, then $1-S(S(n))=P(1-S(n))=P(0)=0$.

$(ii)$ Use Equational Induction on $n$. If $n=0$ then by Lemma \ref{lemmaarithm}.(viii) and Theorem \ref{Theorem1}.(xxvi), $2-S(S(0))=P(2-S(0))=P(P(2-0))=P(P(2))=P(P(S(S(0))))=P(S(0))=0$. Suppose $2-S(S(n))=0$, then $2-S(S(S(n)))=P(2-S(S(n)))=P(0)=0$.

$(iii)$ Let $n\not=0$ and $n\not=1$. Since $n\not=0$, we have $(\exists m)(S(m)=n)$ by successor axiom $(iii)$. Since $S(m)=n\not=1=S(0)$, we have $m\not=0$ and by successor axiom $(iii)$, $(\exists r)(S(r)=m)$. Then $n=S(S(r))$, hence $2-n=2-S(S(r))=0$ by $(ii)$.

$(iv)$ Let $m\not=0$ and $P(n)=m$, then $n\not=0$ and by successor axiom $(iii)$, $(\exists r)(S(r)=n)$. Then by Theorem \ref{Theorem1}.(xxvi), $m=P(n)=P(S(r))=r$ and $S(m)=S(r)=n$. Conversely, if $n=S(m)$ then $P(n)=P(S(m))=m$ by Theorem \ref{Theorem1}.(xxvi).

$(v)$ Use Equational Induction on $n$. If $n=0$ then by Lemma \ref{lemmaarithm}.(viii), $0-0=0$. Suppose $0-n=0$, then $0-S(n)=P(0-n)=P(0)=0$.

$(vi)$ Let $n-1=0$ and suppose that $n\not=0$ and $n\not=1$. Then as in $(iii)$ we have $n=S(S(r))$ for some $r$, so $n-1=P(n)=P(S(S(r)))=S(r)\not=0$ by successor axiom $(i)$, contradiction. Conversely, if $n=0$ then $0-1=0$ by $(v)$, if $n=1$ then $1-1=P(1-0)=P(1)=0$.

$(vii)$ If $n=0$ then $1-0=1\not=0$. If $n\not=0$, then by successor axiom $(iii)$, $(\exists m)(S(m)=n)$. So $1-n=1-S(m)=0$ by $(i)$.

$(viii)$ If $n=1$ then $(n-1)+(1-n)=0+0=0$ by $(vi)$ and $(vii)$. Suppose $(n-1)+(1-n)=0$ then by Theorem \ref{Theorem1}.(vi), $n-1=0$ and $1-n=0$. By $(vi)$, $n=0\vee n=1$. By $(vii)$, we have $n\not=0$. Hence $n=1$.

$(ix)$ If $n=1$ then $(2-1)\cdot1=(S(S(0))-S(0))\cdot 1=P(S(S(0)-0))\cdot1=P(S(S(0)))\cdot1=S(0)\cdot1=1$. If $n\not=1$ then by $(iii)$, $n=0\vee2-n=0$, hence by Theorem \ref{Theorem1}.(xi), $(2-n)\cdot n=0$.

$(x)$ Note that $n$ is a variable of sort $\omega$ and $m$ is an external natural number. If $n=m^*$ then by Definition \ref{definition*}, $n=S(...S(0))$, where there are $m$ $S$'s. By Theorem \ref{Theorem1}.(xxvi), $P(...P(n))=S(0)=1$, since there are $m-1$ $P$'s.

To prove the converse we use external induction on $m\geq2$. The basis case is $P(n)=1\rightarrow n=2$ with $m=2$. To prove this suppose $P(n)=1$. Then $n\not=0$ and $n\not=1$ and as in $(iii)$ we have that $n=S(S(r))$ for some $r$. Then $P(n)=P(S(S(r)))=S(r)$, so $S(r)=P(n)=1=S(0)$. By successor axiom $(ii)$, $r=0$. Then $n=S(S(0))=2$.

The induction step. Suppose $P(...P(n))=1\rightarrow n=m^*$, $m\geq2$ is provable in $PRA_{fcn}$, where there are $m-1$ $P$'s. Then by substituting $P(n)$ for $n$, $P(...P(P(n)))=1\rightarrow P(n)=m^*$ is provable in $PRA_{fcn}$. By $(iv)$, $P(n)=m^*\leftrightarrow n=S(m^*)=(m+1)^*$. Then $P(...P(n))=1\rightarrow n=(m+1)^*$ is provable in $PRA_{fcn}$, where there are $m$ $P$'s.

$(xi)$ Let $\varphi$ be a propositional combination of equations $t_1=n^*_1$, $t_2=n_2^*$,..., $t_k=n_k^*$ for some $k$. For all $i$, $1\leq i\leq k$, if $n^*_i=0$ then by Theorem \ref{Theorem1}.(xiv), $t_i=0$ is replaced by $\overline{sg}(t_i)=1$, if $n^*_i=1$ then $t_i=1$ remains unchanged, if $n_i^*\geq2$ then by $(x)$, $t_i=n_i^*$ is replaced by $P(...P(t))=1$.

$(xii)$ It is proved by external induction on every propositional combination of equations of the form $t=n^*$. By $(xi)$ we assume that they are propositional combinations of equations of the form $t=1$.

The basis case $t=1$. By $(viii)$, $t=1\leftrightarrow(n-1)+(1-n)=0$.

Suppose $\varphi\leftrightarrow t=0$ is provable, then $\neg\varphi\leftrightarrow t\not=0\leftrightarrow\overline{sg}(t)=0$.

Suppose $\varphi\leftrightarrow s=0$ and $\psi\leftrightarrow t=0$ are provable, then $\varphi\wedge\psi\leftrightarrow s+t=0$ is provable by Theorem \ref{Theorem1}.(vi).

$(xiii)$ Let $\varphi$ be a singular open formula. By $(xii)$, $\varphi$ is equivalent to an equation $t=0$. By Lemma \ref{lemma3} $PRA_{fcn}$ proves induction for the equation $t=0$, i.e. induction for $\varphi$.

$(xiv)$ $P(0)=0$, $P(1)=P(S(0))=0$. Suppose that $n\not=0$ and $n\not=1$. Then as in $(iii)$ we have that $n=S(S(m))$ for some $m$. Then $P(n)=P(S(S(m)))=S(m)\not=0$ by successor axiom $(i)$.

$(xv)$ Use Equational induction on $m$. If $m=0$ then $S(n)-S(0)=P(S(n)-0)=P(S(n))=n=n-0$. Suppose $S(n)-S(m)=n-m$, then $S(n)-S(S(m))=P(S(n)-S(m))=P(n-m)=n-S(m)$.

$(xvi)$ Use Equational induction on $m$. If $m=0$ then $(n+0)-0=n$. Suppose $(n+m)-m=n$, we want to prove $(n+S(m))-S(m)=n$. If $m=0$, then $(n+S(0))-S(0)=P((n+S(0))-0)=P(n+S(0))=P(S(n+0))=P(S(n))=n$. If $m\not=0$, then by by successor axiom $(iii)$, $(\exists m')(m=S(m'))$. We have $(n+m)-m=n$, i.e. $(n+S(m'))-S(m')=n$, then $S(n+m')-S(m')=n$. Hence by $(xv)$, $S(S(n+m'))-S(S(m'))=n$. Hence $(n+S(S(m')))-S(S(m'))=n$, i.e. $(n+S(m))-S(m)=n$.

$(xvii)$ $n+r=m+r\rightarrow(n+r)-r=(m+r)-r\rightarrow n=m$ by $(xvi)$.

$(xviii)$ We use Theorem \ref{Theorem1}.(xiii). If $r=0$ then $sg(r)=0$, then $n\cdot sg(r)=n\cdot0=0=m\cdot0=m\cdot sg(r)$. If $r\not=0$ and $n=m$ then $sg(r)=1$ and $n\cdot sg(r)=n=m=m\cdot sg(r)$. Conversely, assume $n\cdot sg(r)=m\cdot sg(r)$. If $r=0$, then we are done, if $r\not=0$, then $sg(r)=1$ and so $n=m$.

$(xix)$ Use Equational induction on $n$. If $n=0$, then $0-0=0$ by Lemma \ref{lemmaarithm}.(viii). Suppose $n-n=0$, then by $(xv)$, $S(n)-S(n)=n-n=0$.

$(xx)$ Use Equational induction on $n$. If $n=0$, then $S(0)-0=S(0)=1$. Suppose $S(n)-n=1$, then by $(xv)$, $S(S(n))-S(n)=S(n)-n=1$.

$(xxi)$ Let $n\not=0$. Then by successor axiom (iii), $(\exists r)(S(r)=n)$. Then $P(n)+S(m)=P(S(r))+S(m)=r+S(m)=r+(m+1)=r+(1+m)=(r+1)+m=S(r)+m=n+m$ by Theorem \ref{Theorem1}.(i), Theorem \ref{Theorem1}.(iv), Theorem \ref{Theorem1}.(v).

$(xxii)$ By $(xviii)$ this statement is equivalent to an equation in $m$ and $n$. So we can use Equational Induction on $m$ to prove the statement. If $m=0$ then $(n-0)+0=n$. Suppose $n-m=0\vee(n-m)+m=n$. We need to prove that $n-S(m)=0\vee(n-S(m))+S(m)=n$. If $n-m=0$ then $P(n-m)=0$, then $n-S(m)=P(n-m)=0$ by Lemma \ref{lemmaarithm}.(viii). If $n-m\not=0$ and $(n-m)+m=n$, then by $(xxi)$, $(n-S(m))+S(m)=P(n-m)+S(m)=(n-m)+m=n$.

$(xxiii)$ The formula can be put into the form $n-m\not=0\vee m-n\not=0\vee n=m$, where $n-m\not=0\vee m-n\not=0$ is a singular open formula. Hence by $(xii)$ we can put the formula into the form $t=0\vee n=m$ and then by $(xviii)$ into a single equation, and after that use Equational Induction on $m$.

The basis case $m=0$. If $n-0=0-n=0$ then $n=0$ and $n=m$ holds.

Suppose $n-m=m-n=0\rightarrow n=m$ and assume $n-S(m)=S(m)-n=0$. Clearly, $n\not=0$ (otherwise $S(m)-0=0$, contradiction to successor axiom $(i)$). By successor axiom $(iii)$, $n=S(r)$ for some $r$, hence $S(m)-n=S(m)-S(r)=0$. By $(xv)$, $S(m)-S(r)=m-r=0$. By Lemma \ref{lemmaarithm}.(viii), $m-n=m-S(r)=P(m-r)=P(0)=0$. If $n-m=0$ then by induction hypothesis $n=m$, then by $(xx)$, $1=S(m)-m=S(m)-n=0$, contradiction. Hence $n-m\not=0$. Since $n-S(m)=0$, we have $P(n-m)=n-S(m)=0$, then by $(xiv)$, $n-m=0\vee n-m=1$. Since $n-m\not=0$, we have $n-m=1$. Also by $(xxii)$, $(n-m)+m=n$. Then $1+m=n$ and by Theorem \ref{Theorem1}.(i) we have $n=S(m)$.

$(xxiv)$ Follows from $(xxiii)$ and $(xix)$.

$(xxv)$ Suppose $n-m\not=0$ and $m-n\not=0$. By $(xxii)$, $(n-m)+m=n$ and $(m-n)+n=m$. Hence $(n-m)+m+(m-n)+n=n+m$. By Theorem \ref{Theorem1}.(iv) and Theorem \ref{Theorem1}.(v), $(n-m)+(m-n)+m+n=0+m+n$. By $(xvii)$, $(n-m)+(m-n)=0$. By Theorem \ref{Theorem1}.(vi), $n-m=0$ and $m-n=0$.

$(xxvi)$ Suppose $n-m=0$. If $m-n=0$ then by $(xxiii)$, $n=m$ and we can take $r=0$. Assume $m-n\not=0$. By $(xxii)$, $(m-n)+n=m$ and we can take $r=m-n$.

$(xxvii)$ Let $\varphi$ be a propositional combination of equations $s_1=t_1$,...$s_k=t_k$ for some $k$. By $(xxiv)$, $\varphi$ is provably equivalent to a propositional combination of equations $r_i=0$, i.e. to a singular open formula. By $(xii)$, $\varphi$ is provably equivalent to an equation $t=0$.

$(xxviii)$ By $(xxvii)$ and $(xiii)$.
\end{proof}

\

Induction for open formulas from Lemma \ref{induction}.(xxviii) will be referred to as open induction.

The following lemma let us to define functions by conditional terms.

\

\begin{lemma}[Friedman \cite{Fr21}]\label{condition} In $PRA_{fcn}$ we can define functions
$$
f(n,m,r)=\left\{
  \begin{array}
  [c]{ll}%
  s, \mbox{ if } \varphi
   \\
  t, \mbox{ otherwise}
  \end{array}
\right.
$$

with extensional uniqueness, where $s,t$ are terms and $\varphi$ is an open formula.
\end{lemma}
\begin{proof}

By definition of terms in $L_{fcn}$ and composition axioms, the term $s$ can be rewritten as $g(n,m,r)$ and the term $t$ can rewritten as $h(n,m,r)$ for some ternary functions $g$ and $h$. By Theorem \ref{induction}.(xxvii), $\varphi$ is equivalent to an equation $f_1(n,m,r)=0$.

Define the auxiliary function $f_2(n,m,r)=n\cdot\overline{sg}(m)+r\cdot sg(m)$. By Theorem \ref{Theorem1}.(xiii) and Theorem.\ref{Theorem1}.(xiv), $f_2(n,m,r)=n$ if $m=0$ and $f_2(n,m,r)=r$ otherwise.

By composition axioms define $f(n,m,r)=f_2(g(n,m,r),f_1(n,m,r),h(n,m,r))$.

Extensional uniqueness follows from the Theorem \ref{induction}.(xxviii).
\end{proof}

\

The next step is to define $<$ in $PRA_{fcn}$ that is done by using the function from the following lemma.

\

\begin{lemma}[Friedman \cite{Fr21}]\label{lemleq} $PRA_{fcn}$ proves the existence and extensional uniqueness of the function $f$ such that $(\forall m)(f(m,0)=0)\wedge(\forall m,n)((f(m,S(n))=S(0)\leftrightarrow f(m,n)=S(0)\vee m=n)\wedge(f(m,S(n))=0\leftrightarrow f(m,n)\not=S(0)\wedge m\not=n))$.
\end{lemma}
\begin{proof}
By Lemma \ref{condition} define the auxiliary function $h$:

$$
h(n,m,r)=\left\{
  \begin{array}
  [c]{ll}%
  S(0), \mbox{ if } r=S(0)\vee n=m
   \\
  0, \mbox{ otherwise}
  \end{array}
\right.
$$

By $PRA$, $(\exists f)(\forall m)(f(m,0)=0)\wedge(\forall m,n)(f(m,S(n))=h(m,n,f(m,n))$.

The extensional uniqueness is by open induction.
\end{proof}

\begin{definition}[Friedman \cite{Fr21}]\label{defleq}

\

$m<n$ if and only if $f(m,n)=S(0)$ for $f$ from Lemma \ref{lemleq}.

$m\leq n$ if and only if $m<n\vee m=n$.
\end{definition}

\begin{lemma} [Friedman \cite{Fr21}]\label{leq} The following are provable in $PRA_{fcn}$.

(i) $\neg m<0$

(ii) $m<S(n)\leftrightarrow m\leq n$

(iii) Every propositional combination of inequalities $\leq,<,=$ is equivalent to an equation $t=0$

(iv) Induction holds for propositional combination of inequalities $\leq,<,=$

(v) $0\leq n$

(vi) $m\leq n\leftrightarrow m-n=0\leftrightarrow (n-m)+m=n\leftrightarrow(\exists r)(m+r=n)$

(vii) $m<n\leftrightarrow S(m)<S(n)\leftrightarrow S(m)\leq n$

(viii) $\leq$ is reflexive, connected, transitive, antisymmetric linear ordering with least element $0$

(ix) $<$ is an irreflexive linear ordering with least element $0$, where each $S(n)$ is the immediate successor of $n$, each $n\not=0$ is the immediate successor of $P(n)$

(x) $n+m\leq n+r\leftrightarrow m\leq r$

(xi) $n+m<n+r\leftrightarrow m<r$
\end{lemma}
\begin{proof}

\

$(i)$ By Definition \ref{defleq}, $m<0$ if and only if $f(m,0)=S(0)$, but by Lemma \ref{lemleq}, $f(m,0)=0$ for all $m$.

$(ii)$ By Definition \ref{defleq} and Lemma \ref{lemleq}, $m<S(n)\leftrightarrow f(m,S(n))=S(0)\leftrightarrow f(m,n)=S(0)\vee m=n\leftrightarrow m<n\vee m=n\leftrightarrow m\leq n$.

$(iii)$ By Definition \ref{defleq} and Theorem \ref{induction}.(xii), $m<n\leftrightarrow f(m,n)=S(0)\leftrightarrow t=0$ for some $t$. Hence by Theorem \ref{induction}.(xii) every propositional combination of inequalities $\leq,<,=$ is equivalent to an equation $t=0$.

$(iv)$ This follows from $(iii)$ and open induction.

$(v)$ Use induction on $n$. If $n=0$ then we have $0=0$ and by Definition \ref{defleq}, $0\leq0$. Suppose $0\leq n$, then by $(ii)$, $0<S(n)$, hence by Definition \ref{defleq}, $0\leq S(n)$.

$(vi)$ We show $m\leq n\rightarrow m-n=0\rightarrow (n-m)+m=n\rightarrow(\exists r)(m+r=n)\rightarrow m\leq n$.

First, prove $m\leq n\rightarrow m-n=0$ by induction on $n$. Let $n=0$ and assume $m\leq0$, then by $(i)$, $m=0$ and $m-n=0-0=0$, so $m\leq0\rightarrow m-0=0$. Suppose $m\leq n\rightarrow m-n=0$ and assume $m\leq S(n)$. If $m=S(n)$ then by Theorem \ref{induction}.(xix), $m-S(n)=S(n)-S(n)=0$, i.e. $m\leq S(n)\rightarrow m-S(n)=0$. If $m<S(n)$ then by $(ii)$, $m\leq n$ and hence $m-n=0$. Then $m-S(n)=P(m-n)=P(0)=0$, i.e. $m=S(n)\rightarrow m-S(n)=0$.

Prove $m-n=0\rightarrow (n-m)+m=n$. Suppose $m-n=0$. By Theorem \ref{induction}.(xxii), $n-m=0\vee(n-m)+m=n$. If $n-m=0$ then by Theorem \ref{induction}.(xxiii), $m=n$ and $(n-m)+m=(n-n)+n=0+n=n$. If $(n-m)+m=n$ then we are done.

To prove $(n-m)+m=n\rightarrow(\exists r)(m+r=n)$ set $r=n-m$ and use Theorem \ref{Theorem1}.(iv).

Now prove $(\exists r)(m+r=n)\rightarrow m\leq n$. Assume $(\exists r)(m+r=n)$. By Theorem \ref{induction}.(xxvi), $m-n=0$. We prove $m-n=0\rightarrow m\leq n$ by induction on $n$. If $n=0$ then assume $m-0=0$, and we have $m=0$, $0\leq 0$. Suppose $m-n=0\rightarrow m\leq n$ and assume $m-S(n)=0$. Then $P(m-n)=0$ and by Theorem \ref{induction}.(xiv), $m-n=0\vee m-n=1$. If $m-n=0$ then $m\leq n$ and by $(ii)$, $m<S(n)$, i.e. $m\leq S(n)$. If $m-n=1$ then by Theorem \ref{induction}.(xxii), $m=(m-n)+n=1+n=S(n)$, i.e. $m\leq S(n)$.

$(vii)$ We show $m<n\rightarrow S(m)<S(n)\rightarrow S(m)\leq n\rightarrow m<n$.

Suppose $m<n$, then $m\leq n$ and by $(vi)$, let $m+r=n$ with $r\not=0$. Then by successor axiom $(iii)$, $(\exists t)(r=S(t))$ for some $t$ and $n=m+r=m+S(t)=m+t+1=m+1+t=S(m)+t$. Hence $S(m)\leq n$ by $(vi)$ and $S(m)<S(n)$ by $(ii)$.

Suppose $S(m)<S(n)$ then $S(m)\leq n$ by $(ii)$.

Suppose $S(m)\leq n$. Then $S(m)+r=n$ for some $r$ by $(vi)$. We have $n=S(m)+r=m+1+r=m+r+1=m+S(r)$. Then $m\leq n$ by $(vi)$. If $m=n$ then $m=S(m)+r=r+S(m)$, then by Theorem \ref{induction}.(xix) and Theorem \ref{induction}.(xx), $0=m-m=r+S(m)-m=r+1=S(r)$, contradiction to successor axiom (i). Hence $m<n$.

$(viii)$ By Theorem \ref{induction}.(xix), $n-n=0$. Then by $(vi)$, $n\leq n$.

By Theorem \ref{induction}.(xxv), $m-n=0\vee n-m=0$. Then by $(vi)$, $m\leq\vee n\leq m$, i.e. $\leq$ is connected.

Let $n\leq m\wedge m\leq r$, then by $(vi)$, $n+s=m\wedge m+t=r$, then $n+s+t=r$, then by $(vi)$, $n\leq r$, i.e. $\leq$ is transitive.

Let $m\leq n\wedge n\leq m$, then by $(vi)$, $n+r=m$ and $m+s=n$. Then $n+r+s=n$, and hence $r+s=0$, and by Theorem \ref{Theorem1}.(vi), $r=s=0$ and $m=n$, i.e. $\leq$ is antisymmetric.

$0$ is the least element by $(v)$.

$(ix)$ To show that $<$ is irreflexive, suppose $n<n$. Then by $(vii)$, $S(n)\leq n$, and so by $(vi)$, let $n=S(n)+r=n+1+r=r+1+n$. Then $0=n-n=r+1+n-n=r+1=S(r)$, contradiction to successor axiom $(i)$. So $<$ is irreflexive and from $(viii)$ it is transitive and has trichotomy, with least element $0$.

By $(ii)$, $n<S(n)$. Suppose $n<m<S(n)$ for some $m$, then by $(ii)$, $n<m\wedge(m=n\vee m<n)$ which contradicts the linearity of $<$. Hence each $S(n)$ is the immediate successor of $n$ in $<$.

Let $n\not=0$, then by Theorem \ref{Theorem1}.(xxv), $S(P(n))=n$ is the immediate successor of $P(n)$ in $<$.

$(x)$ Suppose $n+m\leq n+r$, then by $(vi)$, $(\exists t)(n+m+t=n+r)$. Then $m+t=r$, and by $(vi)$, $m\leq r$.

Suppose $m\leq r$, then by $(vi)$, $(\exists t)(m+t=r)$. Then $n+m+t=n+r$, and so $n+m\leq n+r$ by $(vi)$.

$(xi)$ Suppose $n+m<n+r$, then $n+m\leq n+r$ and by $(x)$, we have $m\leq r$. Now $m=r$ is impossible by irreflexivity of $<$, so $m<r$.

Suppose $m<r$, then $m\leq r$ and by $(x)$, we have $n+m\leq n+r$. If $n+m=n+r$ then $m=r$ violating irreflexivity.
\end{proof}

\

Using $\leq$ one can define functions that takes maximum value of given functions on intervals and functions that returns greatest arguments where given functions are maximized on intervals.

\begin{lemma}[Friedman \cite{Fr21}]\label{max} Let $f$ be unary and $g$ be ternary. Assume $(\forall m,n)(g(m,n,0)=0)$. $PRA_{fcn}$ proves that the following functions exist.

(i) $
max(m,n)=\left\{
  \begin{array}
  [c]{ll}%
  n, \mbox{ if } m\leq n
   \\
  m, \mbox{ otherwise}
  \end{array}
\right.
$

(ii) $f_{max}(m,n)=\max \{f(r):r\leq n\wedge g(m,n,r)=0\}$

(iii) $h(m,n)=\max \{r\leq n:f(r)=f_{max}(m,n)\}$
\end{lemma}
\begin{proof}

\

$(i)$ By Lemma \ref{condition} and Definition \ref{defleq}.

$(ii)$ By $PRA$, $(i)$ and Lemma \ref{condition}, define $f_{max}(m,0)=0$ and

$
f_{max}(m,S(n))=\left\{
  \begin{array}
  [c]{ll}%
  max(f_{max}(m,n),f(S(n))), \mbox{ if } g(m,S(n),S(n))=0
   \\
  f_{max}(m,n), \mbox{ otherwise}
  \end{array}
\right.
$

Clearly, $f_{max}(m,n)\leq f_{max}(m,S(n))$. Also $r\leq n\rightarrow f_{max}(m,r)\leq f_{max}(m,n)$ by open induction on $n$.

Now let $g(m,n,r)=0\wedge r\leq m$, then $f(r)\leq f_{max}(m,r)\leq f_{max}(m,n)$, i.e. $f_{max}(m,n)=\max \{f(r)|r\leq n\wedge g(m,n,r)=0\}$.

$(iii)$ By $PRA$ and Lemma \ref{condition}, define $h(m,0)=0$ and

$
h(m,S(n))=\left\{
  \begin{array}
  [c]{ll}%
  S(n), \mbox{ if } f(S(n))=f_{max}(m,n)
   \\
  h(m,n), \mbox{ otherwise.}
  \end{array}
\right.
$
\end{proof}

\

The following lemma shows that $PRA_{fcn}$ proves the existence of the standard pairing system, based on the pairing function $(x^2+2xy+y^2+3x+y)/2$ introduced by Cantor in \cite{Cantor} and expressed by triangular numbers.

\begin{lemma}[Friedman \cite{Fr21}]\label{pair} The following are provable in $PRA_{fcn}$.

(i) There exists extensionally unique unary function $t$ such that $t(0)=0$ and $t(S(n))=t(n)+n+1$

(ii) $m<n\rightarrow t(m)<t(n)$

(iii) $t(0)=0$, $t(1)=1$ and $n\geq 2\rightarrow n<t(n)$

(iv) There exists $t'$ such that $t'(n)=\max\{r:t(r)\leq n\}$

(v) $(\forall n)(\exists!m,r)(n=t(m)+r\wedge r\leq m)$

(vi) There exists surjective function $\langle n,m\rangle=t(n+m)+m$. There exists $p_1$, $p_2$ such that $p_1(\langle n,m\rangle)=n$, $p_2(\langle n,m\rangle)=m$ and $\langle p_1(n),p_2(n)\rangle=n$. The function $\langle,\rangle$ is a bijection.
\end{lemma}
\begin{proof}

\

$(i)$ The function $t$ exists by $PRA$ and is extensionally unique by Equational Induction.

$(ii)$ Note, that since $0<n+1$, we have $t(n)+0<t(n)+n+1=t(S(n))$ by Lemma \ref{leq}.(xi).

We prove $m<n\rightarrow t(m)<t(n)$ by open induction on $n$. Suppose $m<n\rightarrow t(m)<t(n)$ and assume $m<S(n)$. By Lemma \ref{leq}.(ii), $m\leq n$. If $m<n$ then $t(m)<t(n)<t(S(n))$. If $m=n$ then $t(m)=t(n)<t(S(n))$.

$(iii)$ Let $n=m+2$. Since $t(m+1)>0$, we have $t(n)=t(m+2)=t(m+1)+m+2=t(m+1)+n>n$.

$(iv)$ Define by Lemma \ref{condition}

$
g(m,n,r)=\left\{
  \begin{array}
  [c]{ll}%
  0, \mbox{ if } t(r)\leq n
   \\
  1, \mbox{ otherwise}.
  \end{array}
\right.
$

Then apply Lemma \ref{max}.(ii) to get $t_{max}(n)=\max \{t(r):r\leq n\wedge t(r)\leq n\}$. Then by Lemma \ref{max}.(iii), set $t'(n)=\max \{r\leq n:t(r)=t_{max}(n)\}$.

$(v)$ Let $n$ be given and let by $(iv)$, $m$ be greatest such that $t(m)\leq n$. Set $r=n-t(m)$. Since $t(m+1)\leq n$ is false, then $n<t(m+1)=t(m)+m+1$. So $r=n-t(m)<t(m)+m+1-t(m)=m+1=S(m)$, hence by Lemma \ref{leq}.(ii), $r\leq m$.

For uniqueness, let $t(m)+r=t(m')+r'\wedge r\leq m\wedge r'\leq m'$. Assume $m<m'$. Then $S(m)\leq m'$, and by $(ii)$, $t(S(m))\leq t(m')$. Then by Lemma \ref{leq}.(x), $t(S(m))+r'\leq t(m')+r'=t(m)+r$, and so $t(m)+m+1+r'\leq t(m)+r$. Hence $m+1+r'\leq r\leq m$, contradiction. By symmetry, $m'<m$ is impossible. Therefore $m=m'$ and $r=r'$.

$(vi)$ Note, that $\langle,\rangle$ is defined by composition of functions $T$ and $+$. Let $n$ be given, then by $(v)$, $(\exists!m,r)(n=t(m)+r\wedge r\leq m)$. By Lemma \ref{leq}.(vi) $\exists k(r+k=m)$, then $\langle k,r\rangle=t(k+r)+r=t(m)+r=n$, so $n=\langle k,r\rangle$ and $\langle, \rangle$ is a surjective function.

To define $p_2$ note that $t(n+m)\leq t(n+m)+m$ and $t(n+m+1)=t(n+m)+n+m+1$, i.e. $\neg t(n+m+1)\leq t(n+m)+m$. Hence by $(iv)$, $t'(t(n+m)+m)=n+m$. Now define $p_2(t(n+m)+m)=(t(n+m)+m)-t(t'(t(n+m)+m))=(t(n+m)+m)-t(n+m)=m$, so we have $p_2(\langle n,m\rangle)=m$.

Define $p_1(t(n+m)+m)=t'(t(n+m)+m)-p_2(t(n+m)+m)=(n+m)-m=n$, so we have $p_1(\langle n,m\rangle)=n$.

Let $n$ be given, then $n=\langle m,r\rangle$ for some $m,r$, hence $p_1(n)=m$, $p_2(n)=r$. So $n=\langle p_1(n),p_2(n)\rangle$.

Suppose $\langle n,m\rangle=\langle n',m'\rangle$ then $n=p_1(\langle n,m\rangle)=p_1(\langle n',m'\rangle)=n'$ and $m=p_2(\langle n,m\rangle)=p_2(\langle n',m'\rangle)=m'$, so $\langle,\rangle$ is a bijection.
\end{proof}

\begin{definition}[Friedman \cite{Fr21}] A pairing system consists of functions $\langle,\rangle, p_1, p_2$, where $\langle,\rangle$ is binary and $p_1,p_2$ are unary, such that $p_1(\langle m,n\rangle)=m\wedge p_2(\langle m,n\rangle)=n\wedge\langle p_1(n),p_2(n)\rangle=n$.
\end{definition}

\

The following lemma will be crucial for the proof that permutation axiom $PERM$ implies $MIN^1$ from Definition \ref{statements}.

\

\begin{lemma}[Friedman \cite{Fr21}]\label{perm} The following are provable in $PRA_{fcn}$. Let the characteristic function of $A\subseteq2\omega+1$ exist. Let $h:\omega\rightarrow\omega$ map $A$ one-one onto $2\omega+1$.

(i) There is a surjective $f:\omega\rightarrow\omega-A$ such that $(\forall n)(f(n)\geq n)\wedge(\forall m,n)(m<n\rightarrow f(m)<f(n))$

(ii) There exists $f'$ such that for all $n\not\in A$, $f(f'(n))=n$ ($f'$ maps $\omega-A$ one-one onto $\omega$).

(iii) There exists a permutation $g$ of $\omega$ which agrees with $h$ on $A$.
\end{lemma}
\begin{proof}

\

$(i)$ Since the characteristic function $\chi$ of $A\subseteq2\omega+1$ exists, i.e. the condition $n\in A$ could be expressed as $\chi_A(n)=1$, we define by Lemma \ref{condition} the following function:

$
g'(n)=\left\{
  \begin{array}
  [c]{ll}%
  n+1, \mbox{ if } odd(n)=1, \mbox{ or } n+1\not\in A
   \\
  n+2, \mbox{ otherwise}.
  \end{array}
\right.
$

Then $g'(n)$ is the least number outside of $A$ and greater than $n$.

Now define $f(0)=0$, $f(n+1)=g'(f(n))$. Clearly $(\forall n)(f(n)\not\in A)$ and it is easy to prove $n<m\rightarrow f(n)<f(m)$ by open induction.

Since $f$ is strictly increasing, it is also easy to prove $f(n)\geq n$ by open induction.

Let $n\not\in A$. By Lemma \ref{max}.(ii) define $f_{max}(n)=\max \{f(r):r\leq n\wedge f(r)\leq n\}$ and denote $t=f_{max}(n)$. Let $s$ be such that $t=f(s)$ and suppose that $t<n$.

We claim that $g'(t)\leq n$. If $g'(t)=t+1$ then $g'(t)\leq n$ by Lemma \ref{leq}.(vii). Suppose $g'(t)=t+2$ and $n<g'(t)=S(t+1)$. Then by Lemma \ref{leq}.(ii), $n\leq t+1$. By Definition \ref{defleq}, $n<t+1\vee n=t+1$. If $n<t+1$ then by Lemma \ref{leq}.(ii), $n\leq t$, contradicts to $t<n$. If $n=t+1$, then since $n\not\in A$, by definition, $g'(t)=t+1=n$, contradiction. So, we have $g'(t)\leq n$.

By definition, $f(s+1)=g'(f(s))=g'(t)\leq n$. This contradicts to the choice of $t$. Hence $t=n$, i.e. $f(s)=t=n$ and $f$ is surjection.

$(ii)$ Since $n<m\rightarrow f(n)<f(m)$, there is the only $s$ such that $f(s)=n$. By Lemma \ref{max}.(iii) define $f'(n)=s$. Clearly $n\not\in A'\rightarrow f(f'(n))=n$.

$(iii)$ Define

$
g(n)=\left\{
  \begin{array}
  [c]{ll}%
  2\cdot f'(n), \mbox{ if } n\not\in A
   \\
  h(n), \mbox{ otherwise}.
  \end{array}
\right.
$

Then $g$ maps $\omega-A$ one-one onto $2\omega$, and $A$ one-one onto $2\omega+1$. Hence $g$ is a permutation which agrees with $h$ on $A$.
\end{proof}

\

The next lemma shows the existence of the quotient function that gives the result of the division of a number by 2. To prove this we use the $odd$ function and its properties from Theorem \ref{Theorem1}.

\begin{lemma}\label{quotient} The following are provable in $PRA_{fcn}$.

(i) $(\exists q)(q(0)=0\wedge q(S(n))=q(n)+\overline{sg}(odd(S(n))))$

(ii) $odd(S(2\cdot n))=1$

(iii) $q(2\cdot n)=n$
\end{lemma}
\begin{proof}

\

$(i)$ This function exists by $PRA$ and Lemma \ref{lemmaarithm}.

$(ii)$ By Theorem \ref{Theorem1}.(xix), $odd(S(2\cdot n))=\overline{sg}(odd(2\cdot n))=\overline{sg}(0)=1$.

$(iii)$ Use Equational Induction on $n$. If $n=0$ then $q(2\cdot 0)=q(0)=0$. Suppose $q(2\cdot n)=n$, then by $(ii)$, Theorem \ref{Theorem1}.(xii) and Theorem \ref{Theorem1}.(xx), $q(2\cdot S(n))=q(S(n)+S(n))=q(n+1+n+1)=q(n+n+S(0)+S(0))=q(S(2\cdot n+S(0)+0))=q(S(2\cdot n+S(0)))=q(S(S(2\cdot n+0)))=q(S(S(2\cdot n)))=q(S(2\cdot n))+\overline{sg}(odd(S(S(2\cdot n))))=q(2\cdot n)+\overline{sg}(odd(S(2\cdot n)))+\overline{sg}(odd(2\cdot n))=n+\overline{sg}(1)+\overline{sg}(0)=n+0+1=S(n)$.
\end{proof}

\

Now we prove the main theorem.

\

\begin{theorem}\label{Theorem2}

The following assertions are pairwise equivalent over $PRA_{fcn}$.

1. $MIN^3:(\forall m,n)(\exists r)(f(m,n,r)=0)\rightarrow(\exists g)(\forall m,n)(g(m,n)=(\mu r)(f(m,n,r)=0))$.

2. $MIN^2:(\forall m)(\exists n)(f(m,n)=0)\rightarrow(\exists g)(\forall m)(g(m)=(\mu n)(f(m,n)=0))$.

3. $MIN^1:(\forall m)(\exists! n)(f(m,n)=0)\rightarrow(\exists g)(\forall m)(f(m,g(m))=0)$.

4. $PERM:(\forall n)(\exists!m)(f(m)=n)\rightarrow(\exists g)(\forall n)(f(g(n))=n)$
\end{theorem}
\begin{proof}

\

$1\rightarrow2$. Suppose that $(\forall m)(\exists n)(f(m,n)=0)$, then by Lemma \ref{lemma1} define $f'(r,m,n)=f(m,n)$. We have $(\forall r,m)(\exists n)(f'(r,m,n)=0)$. By $MIN^3$ there exists $g$ such that $(\forall r,m)(g(r,m)=(\mu n)(f'(r,m,n)=0))$. By Lemma \ref{lemma1} set $g'(m)=g(r,m)$ for some fixed $r$, so we have $g'$ such that $(\forall m)(g'(m)=(\mu n)(f(m,n)=0))$.

$2\rightarrow3$. Obvious.

$3\rightarrow4$. Suppose that $(\forall n)(\exists!m)(f(m)=n)$. By Lemma \ref{condition} define $h$ as follows:

$$
h(n,m)=\left\{
  \begin{array}
  [c]{ll}%
  0, \mbox{ if } f(m)=n
   \\
  1, \mbox{ otherwise}.
  \end{array}
\right.
$$

We have $(\forall n)(\exists!m)(h(n,m)=0)$, hence by $MIN^1$ there is $g$ such that $(\forall m)(h(m,g(m))=0)$, i.e. $(\exists g)(\forall m)(f(g(m))=m)$.

$4\rightarrow3$ (Friedman \cite{Fr21}). Let $(\forall m)(\exists! n)(f(m,n)=0)$. Let $A=\{2\cdot \langle m,n\rangle+1|f(m,n)=0\}$. Then $A$ has a characteristic function that is defined by conditional terms from Lemma \ref{condition} using the pairing system $\langle,\rangle, p_1,p_2$ from Lemma \ref{pair}. Define $h$ as follows:

$$
h(r)=\left\{
  \begin{array}
  [c]{ll}%
  2\cdot m+1, \mbox{ if } r=2\cdot \langle m,n\rangle+1
   \\
  0, \mbox{ otherwise}.
  \end{array}
\right.
$$

Then $h$ maps $A$ one-one onto $2\omega+1$. By Lemma \ref{perm}.(iii), let $g$ be a permutation of $\omega$ which agrees with $h$ on $A$. Applying PERM, $g^{-1}$ exists and for all $m$, $g^{-1}(2\cdot m+1)=2\cdot \langle m,n\rangle+1$ for the unique $n$ with $f(m,n)=0$. Let $q$ be the quotient function from Lemma \ref{quotient}, define $g'(m)=p_2(q(g^{-1}(2\cdot m+1)-1))=p_2(q(2\cdot \langle m,n\rangle+1-1))=p_2(q(2\cdot \langle m,n\rangle))=p_2(\langle m,n\rangle)=n$. Then $(\forall m)(f(m,g'(m))=0)$.

$3\rightarrow2$ (Friedman \cite{Fr21}). Let $(\forall m)(\exists n)(f(m,n)=0)$. Define $h$ as follows by Lemma \ref{condition}:

$$
h(m,n)=\left\{
  \begin{array}
  [c]{ll}%
  0, \mbox{ if } f(m,n)=0\wedge(\forall r<n)(f(m,n)\not=0)
   \\
  1, \mbox{ otherwise}.
  \end{array}
\right.
$$

Clearly, $(\forall m)(\exists!n)(h(m,n)=0)$. By $MIN^1$ there exists $g$ such that $(\forall m)(h(m,g(m))=0)$, i.e. $f(m,g(m))=0$. Also $(\forall m)(\forall r<g(m))(f(m,g(m))\not=0))$, so $(\forall m)(g(m)=(\mu n)(f(m,n)=0))$.

$2\rightarrow1$ (Friedman \cite{Fr21}). Let $(\forall m,n)(\exists r)(f(m,n,r)=0)$. Since $\langle,\rangle$ is a bijection, define $f'(\langle m,n\rangle,r)=f(p_1(\langle m,n\rangle), p_2(\langle m,n\rangle),r)$. We have $(\forall m,n)(\exists r)(f'(\langle m,n\rangle),r)=0)$. By $MIN^2$ there exists $g'$ such that $(\forall m)(g'(\langle m,n\rangle)=(\mu r)(f'(\langle m,n\rangle),r)=0)$. Define $g(m,n)=g'(\langle m,n\rangle)$.
\end{proof}

Thus, we have presented several equivalent axiomatizations of $ETF$.

\begin{corollary}
$ETF$ is equivalent to $PRA_{fcn}+MIN^k$ and to $COMI_{fcn}+WPRA+MIN^k$ for $k=1,2,3$.
\end{corollary}

\end{document}